\documentclass[reprint,
onecolumn,
superscriptaddress,
nofootinbib,
aps,
pre,
showkeys,
]{revtex4-2}

\usepackage{array}
\usepackage{dcolumn}
\usepackage{amsmath,amssymb,amsthm,graphicx,comment,bm,mathtools}
\usepackage[normalem]{ulem}
\usepackage[dvipsnames]{xcolor}
\usepackage{url}
\usepackage{tikz}
\usepackage{hyperref}
\usepackage{todo}
\usepackage{multirow}

\def\be{\begin{equation}}
\def\ee{\end{equation}}
\def\ben{\begin{eqnarray}}
\def\een{\end{eqnarray}}

\def\op{\hat{\operatorname{P}}}

\newcommand{\sumnj}{\sum_{\substack{n=0 \\ n \ne {\ell_j}}}^k}
\newcommand{\sumnl}{\sum_{\substack{l=0 \\ l \ne {\ell_j}}}^k}
\newtheorem{theorem}{Theorem}
\newtheorem{definition}{Definition}

\newtheorem{property}{Property}

\newtheorem{remark}{Remark}[section]

\newcommand{\Spann}{{\mbox{\rm{Span}}}}
\def\oI{\hat{\operatorname{I}}}
\def\In{{\cal{I}}}

\def\N{\mathbb{N}}

\def\R{\mathbb{R}}

\begin{document}

\title{Recursive construction of biorthogonal polynomials for handling polynomial regression}

\author{Laura Rebollo-Neira} 
\email{l.rebollo-neira@aston.ac.uk}
\affiliation{Department of Applied Mathematics and Data Science, College of Engineering and Physical Sciences, Aston University, B4 7ET, Birmingham, UK}

\author{Jason Laurie}
\email{j.laurie@aston.ac.uk}
\affiliation{Department of Applied Mathematics and Data Science, College of Engineering and Physical Sciences, Aston University, B4 7ET, Birmingham, UK}

\date{\today}

\begin{abstract}
An adaptive procedure for constructing polynomials which are biorthogonal to the basis of monomials in the same finite-dimensional inner product space is proposed. By taking advantage of available orthogonal polynomials, the proposed methodology reduces the well-known instability problem arising from the matrix inversion involved in classical polynomial regression.  The recurrent generation of the biorthogonal basis facilitates the upgrading of all its members to include an additional one. Moreover, it allows for a natural downgrading  of the basis. This convenient feature leads to a straightforward approach for reducing the number of terms in the polynomial regression approximation. The merit of this approach is illustrated through a series of examples where the resulting biorthogonal basis is derived from Legendre, Laguerre, and Chebyshev orthogonal polynomials.
\end{abstract}


\maketitle

\section{Introduction}\label{sec:introduction}

\noindent Polynomial regression is becoming increasingly important for predictive modelling~\cite{Ost12,EO20, AVV22}, data analysis~\cite{MMR17}, signal and images processing~\cite{Im}, machine learning and neuron networks~\cite{SS22, SHL23,MBD22}. Thus, a renewed interest in efficient and effective techniques suitable for tackling common polynomial regression problems has arisen. Fundamentally, polynomial regression problems concern the modelling of a process represented by a function $f$ by  a polynomial $f_k$ of degree $k$, i.e., 
\begin{align*}
	f_k(x)= c_0 + c_1 x + \cdots + c_k x^k,\quad \text{for $x \in \In\subseteq \R$},
\end{align*}
and finding the coefficients $c_n$ for $n=0,\dots,k$ by minimising the $L^2$-norm of the approximation error $\| f - f_k\|_2$.  Many applications using polynomial regression, e.g. such as the method of moments used in Ref.~\cite{MMR17}, approach the problem through the traditional method of solving the least squares equation via the normal equations. However, such an approach requires the inversion of a Gram matrix of Hankel-type which are notorious for being very badly conditioned~\cite{ZT01}, and can only provide a reliable solution for low-order polynomial approximations. This is unfortunate, because the Weierstrass approximation theorem states that  every continuous function defined on a compact interval can be uniformly approximated  by a polynomial up to any desired precision, even if the degree of the polynomial may be high. Moreover, because monomials on a compact interval are among the simplest functions to evaluate, integrate, and differentiate,  polynomial approximations have enormous practical relevance in terms of both analytical and numerical computations of integrals and derivatives. This article introduces an adaptive biorthogonalisation technique which, by leveraging the orthogonality of classic orthogonal polynomials (OP), avoids the need for a  matrix inversion and provides the coefficients of the polynomial approximation by simple  computation of inner products. Additional advantages stem from the recursive nature of the approach that provides an efficient strategy for reducing the number of terms of the polynomial approximation i.e., it is convenient to iteratively increase or decrease the number of terms in the polynomial approximation to adjust the residual error. 

The general theory of biorthogonal polynomials sequences (BPS) is not as extensively developed as the theory of orthogonal polynomials sequences (OPS). The initial research on the topic was conducted by Szegi\"o~\cite{Sze82}, Iserles and Nørsett~\cite{IN88, IN89,IN95}, and Brezinski \cite{Bre92}. Since then, there have been few applications or further studies on BPS. One of the reasons for the extensive use of OP is the convenience of the orthogonality property for their derivation and applications. Another is the stability of the representation, which follows from the orthogonality property.  However, when pursuing the construction of $K$-term polynomial approximations, releasing the condition of orthogonality may render approximations of equal or better quality involving fewer terms. We believe that the limited current applicability of BPS might be partly due to the absence of convenient methodologies for their practical construction. Therefore, our objective is to propose new methodologies for building BPS and enhance the comprehension that may expand the range of potential applications across diverse domains.

Since several definitions of biorthogonal polynomials (BP) exist, let us specify the definition we refer to here: Two polynomial bases $\{p_n\}_{n=0}^k$ and $\{q_n\}_{n=0}^k$ are biorthogonal on an interval ${\cal{I}}\subseteq \R$  if the inner product
\begin{align*}
	\left\langle  p_n ,  q_m\right\rangle =\int_{\mathcal{I}}\, p_n(x)\, q_m(x) \, d\mu = \delta_{m,n},
\end{align*}
where $d\mu= w(x)\, dx$ is a measure with $w(x)$ an acceptable weight function, and $\delta_{m,n}$ is the Kronecker delta-function. More specifically, we refer to \textit{classic-like biorthogonal polynomials} to be a set of polynomials $\left\{\beta_n^{k}\right\}_{n=0}^k$  biorthogonal to the set of monomials $\{x^n\}_{n=0}^k$, with the additional property that $\Spann\left(\{x^n\}_{n=0}^k\right)= \Spann\left(\{\beta_n^{k}\}_{n=0}^k\right)$ for each $k\in \N$. Here, the superscript of $\beta_n^{k}$ indicates that all the elements $n=0,\dots,k$  must be re-calculated if the number $k$ of elements in the basis changes. This is a major difference with OP for if the degree $k$ is increased, say to $k+1$, only a single new independent element needs to be added to the original orthogonal set, without modification of the previously calculated elements. 

While recursive equations for generating biorthogonal functions were proposed in Refs.~\cite{RNL02,LRN02,LRN07}, the general nature of that framework misses the main advantage of dealing with the particular case of monomials, namely, the  wealth of available results concerning the construction of OP. The central contribution of this work is to produce an adaptive methodology for constructing classic-like BP giving rise to representations of orthogonal projections for the subspace spanned by the monomials. In particular, we illustrate the method by deriving BP based on orthogonal bases built from Legendre, Laguerre, and Chebyshev polynomials. 

The paper is organised as follows: Sec.~\ref{sec:preliminary} provides the necessary definitions and background material for the construction of classic-like BP. Sec.~\ref{sec:biorthogonal} proves the biorthogonal property of the proposed construction as well as their convenient adaption for upgrading and downgrading orthogonal projections. Examples involving biorthogonality with respect to particular weight functions and domains are also produced in this section. The conclusions and final remarks are presented in Sec.~\ref{sec:conclusions}.

\section{Preliminary Considerations}\label{sec:preliminary}

\noindent In this section we introduce the necessary background for a self-contained derivation of the proposed construction of BP. Since our purpose is to develop a methodology suitable for practical applications, we consider only biorthogonal representations of orthogonal projections, and therefore,  throughout the paper we restrict ourselves to consider only \textit{finite}-dimensional inner product spaces,  which we denote as $V(\In, w)$ to indicate that the inner product $\left\langle f,g\right\rangle$ between two real functions in $V(\In, w)$ is defined as 
\begin{align*}
\left\langle  f, g\right\rangle= \int_{\In} f(x)\, g(x)\, w(x)\, dx,
\end{align*}
where $w(x)$ is a weight function which satisfies $w(x) \geq 0$, for all $x \in \In$. For $k\in \N$, the subspace $V_k(\In, w) \subseteq V(\In, w)$ is specifically defined as the vector space spanned by the set of monomials $\{x^n\}_{n=0}^k$, i.e.
\begin{align*}
	V_k= \Spann\left(\{x^0,x,\ldots,x^k\}\right), \quad \text{for $x \in \In$}.
\end{align*}

\subsubsection*{Type A and B orthonormal polynomial sequences}

We introduce two definitions for the construction of orthonormal polynomial sequences (OPS). These two types, type A and type B, are general forms based on formulas used to construct sequences of orthogonal polynomials. Before, we define the two types, let us highlight a clarification of our notation. Unless otherwise specified, throughout the paper, superscripts are used as labels. The only exception is for the variable $x$ where the superscript indicates the variable $x$ raised to a power $i\in \N \cup\{0\}$, i.e., $x^i$.

\begin{definition}[Type A Orthonormal Polynomial Sequence]
A type A orthonormal polynomial sequence is a sequence of functions generated by
\begin{align}\label{eq:OPS}
p_n(x)= \sum_{i=0}^n a_i^n x^i,\quad \text{with $a_n^n \neq 0$,}
\end{align}
such that  $\left\langle  p_n, p_m\right\rangle = \delta_{n,m}$.
\end{definition}
For the convenience of some particular cases we differentiate the above definition with a second type, type B of OPS.

\begin{definition}[Type B Orthonormal Polynomial Sequence]
A type B orthonormal polynomial sequence is a sequence of functions generated by 
\begin{align}\label{eq:OPSB}
	p_n(x)= \sum_{i=0}^{\lfloor n/2 \rfloor} a_i^n x^{n-2i},\quad \text{with $a_0^n \neq 0$},
\end{align}
that satisfy $\left\langle p_n, p_m\right\rangle = \delta_{n,m}$. Here $\lfloor n/2 \rfloor$ indicates the integer part of $n/2$. 
\end{definition}

\noindent The polynomials $\left\{p_n(x)\right\}_{n=0}^k$ of Type A or B, form an orthonormal set that span the same subspace as the set of monomials $\{x^n\}_{n=0}^k$ with respect to the inner product in $V(\In,w)$ so that
\begin{align*}
	V_k=\Spann\left(\{x^n\}_{n=0}^k\right)= \Spann\left(\{p_n\}_{n=0}^k\right) \subseteq V(\In,w).
\end{align*}

\begin{property}
If $n\leq m$, then an OPS satisfies
\begin{align}\label{eq:ipxp}
	 \left\langle x^n, p_m\right\rangle= \begin{cases}\displaystyle 
\frac{\delta_{n,m}}{a_m^m} & \text{if $p_m$ is Type A,}\\
\displaystyle\frac{\delta_{n,m}}{a_0^m} & \text{if $p_m$ is Type B}.
\end{cases}
\end{align}
\end{property}

\begin{proof}  Since $x^n \in \Spann\left(\{p_i\}_{i=0}^n\right)$ it can be expressed in terms of the first $n$ orthonormal polynomials as 	
	\begin{align}\label{eq:monomial}
		x^n=\sum_{i=0}^{n} d_i^n p_i(x),
	\end{align}
	where $d_i^n \in \R$ for $i=0, \dots, n$. Hence, for any $m$ satisfying $n<m$ we have that $\left\langle x^n, p_m\right\rangle =0$ for both Type A and Type B OPS. For $n=m$, if $p_m$ is of Type A, using Definition~\eqref{eq:OPS}, we can express 
\begin{align}\label{eq:xp}
	a_m x^m= p_m- \sum_{i=0}^{m-1} a_i^m x^i
\end{align}
so that, from the previous case, it follows that $\left\langle x^m, p_m\right\rangle =1/{a_m^m}$. 

Similarly, if $p_m$ is of Type B, with $n=m$, then using Definition~\eqref{eq:OPSB} we have 
\begin{align}\label{eq:xpb}
	a_0^m x^m= p_m- \sum_{i=0}^{\lfloor m/2 \rfloor}a_i^m x^{m-2i},
\end{align}
so then from Eq.~\eqref{eq:monomial}, we have $\left\langle x^m, p_m\right\rangle= 1/{a_0^m}$. This proves Property 1.
\end{proof}	

\subsubsection*{Orthogonal Projections}

In what follows, it is important to recall some properties of subspaces that allows us to use results regarding orthogonal projections.

\begin{definition}[Direct sum of subspaces]
A vector space $V$ is said to be the \textit{direct sum} of two subspaces, say $V_k$ and $U$, written as
\begin{align*}
	V = V_k \oplus U
\end{align*}
if each $f \in V$ has a unique representation
\begin{align*}
	f= f_k + g, \quad \text{where $f_k \in V_k$ and $g \in U$}.
\end{align*}
\end{definition}
If $V$ is an inner product space $V(\In,w)$ a case of  much interest concerns complementary subspaces which are orthogonal. 
\begin{definition}[Orthogonal complement subspace]
A vector space $U$ is said to be the {\em {orthogonal complement}} of $V_k$ in $V(\In,w)$  if
\begin{align}\label{eq:orthc}
U= V_k^\perp=\left\{g \in V \ :\  \left\langle g, f_k \right\rangle = 0,\forall f_k \in V_k \right\}.
\end{align}
\end{definition}
\begin{definition}[Orthogonal projector]\label{prop:projector}
	Let two vectors spaces $V_k$ and  $V_k^\perp$ be orthogonal complements of each other. Then, given any $f\in V$, the map $\op_{V_k}: V \to V_k$ defined as 
	\begin{align*}
		\op_{V_k}f = f_k \in V_k,
	\end{align*}
	is called the \textit{orthogonal projection} (or orthogonal projector) of $V$ onto $V_k$ if it has the following properties:
	\begin{itemize}
		\item [(i)] $\op_{V_k}$ is idempotent, i.e., $\op^2_{V_k}=\op_{V_k}$.
		\item [(ii)] If $V=V_k \oplus V_k^\perp$ then $\op_{V_k} f_k = f_k$ if $f_k \in V_k$   and $\op_{V_k} g=0$ if $g \in V_k^\perp.$
	\end{itemize}
	\end{definition}
As a consequence of the property of Definition~\ref{prop:projector}(ii) we can express $V_k^\perp= \oI_V V- \op_{V_k} V=  (\oI_V- \op_{V_k}) V=\op_{V_k^\perp} V$, where $\oI_V$ is the identity operator in $V$. Subsequently, it follows that  every $f\in V$ can be decomposed in terms of two contributions $f= \op_{V_k} f +  \op_{V_k^\perp} f$.
\begin{property}[Self-adjoint operator]\label{prop:self}
It readily follows that $\op_{V_k}$ satisfies
\begin{align*}\left\langle f, \op_{V_k} g \right\rangle= \left\langle \op_{V_k}  f, g \right\rangle.
\end{align*}
An operator fulfilling the above property is said to be \textit{self-adjoint}.
\end{property}

The next two theorems are well-known fundamental results concerning the optimal approximation in an inner product subspace, also known as the \textit{least squares approximation}. For pedagogical reasons, we present the proofs tailored to our framework in Appendices A and B.
\begin{theorem}\label{thm:orthp}
Let $V(\In,w)$ be an inner product space and $V_k$ a finite-dimensional subspace of $V(\In,w)$. Then, the least squares approximation in $V_k$ of any $f \in V$ can be obtained by the projection of $f$ onto the subspace $V_k$, i.e., $\op_{V_k} f$.
\end{theorem}
\begin{proof} The proof is presented in Appendix A.\end{proof}
The next theorem prescribes a method of constructing an orthogonal projector onto the finite-dimensional subspace $V_k=\Spann\left(\{x^n\}_{n=0}^k\right)$,  with respect to the inner product defined on $V(\In, w)$. 
\begin{theorem}\label{thm:oop}
Let $\{p_n\}_{n=0}^k$ be an orthonormal basis for the inner product subspace $V_k(\In, w)$. Then the orthogonal projector operator $\op_{V_k}: V \to V_k$ can be expressed as
\begin{align}\label{eq:opob}
\op_{V_k} =\sum_{i=0}^k p_i \left\langle \cdot, p_i\right\rangle,
\end{align}
where $\left\langle \cdot, p_i\right\rangle$ performs an inner product with the basis polynomial $p_i$, i.e., for any $f\in V$ 
\begin{align*}
	\op_{V_k}f =\sum_{i=0}^k p_i \left\langle f, p_i\right\rangle.
\end{align*}
\end{theorem}
\begin{proof} The proof is presented in Appendix B. \end{proof}

\section{Biorthogonal representation of orthogonal projectors}\label{sec:biorthogonal}
The next theorem establishes two distinctive features of the proposed BPS. (i) Their construction readily follows from known OPS of either Type A or B. (ii) The BPS provide us with the  biorthogonal representation for the orthogonal projection operator $\op_{V_k}$.
\begin{theorem}\label{thm:biorth}	
 Let $\{p_n\}_{n=0}^k$ be a finite OPS. Then the orthogonal projector operator $\op_{V_k}$ can be represented as
 \begin{align}\label{eq:opb}
 \op_{V_k}= \sum_{i=0}^k x^i \left\langle \cdot, \beta_i^k \right\rangle,
 \end{align}
where the functions $\beta_n^k(x)$ are defined as
\begin{itemize}
	\item[(i)] \begin{align}\label{eq:bet}
		\beta_n^k(x)= \sum_{j=n}^k a_n^j\ p_j(x),\,x \in \In,
		\end{align}
	if $\{p_n\}_{n=0}^k$ is an OPS of Type A as defined in Eq.~\eqref{eq:OPS}, or 
\item[(ii)] \begin{align}\label{eq:betb}
	\beta_n^k(x)= \sum_{i=0}^{\lfloor (k-n)/2 \rfloor} a_i^{2i+n}\ p_{2i+n} (x),
\end{align}
if $\{p_n\}_{n=0}^k$ is  an OPS of Type B as defined in Eq.~\eqref{eq:OPSB}.
\end{itemize}
\end{theorem}
\begin{proof}

For $n=k=0$ we have $\beta_0^0= p_0$ and $x^0=p_0$. Hence, operator $\op_{V_0}$ as given by Eq.~\eqref{eq:opb} for $k=0$ is the orthogonal projection operator onto $V_0$. Assuming that $\op_{V_k}$ in Eq.~\eqref{eq:opb} is the orthogonal projection operator onto $V_k$ we will show that then $\op_{V_{k+1}}$ is the orthogonal projection operator onto $V_{k+1}$.

Case (i).  Since ${\displaystyle{\beta_n^{k+1} = \beta_n^{k}+a_n^{k+1} p_{k+1}}}$, we have
\begin{align*}
\sum_{n=0}^{k+1} x^n \left\langle \cdot, \beta_n^{k+1} \right\rangle&=\sum_{n=0}^{k} x^n \left\langle \cdot, \beta_n^{k+1} \right\rangle+ x^{k+1}\left\langle \cdot, \beta_{k+1}^{k+1}\right\rangle \\
&= \sum_{n=0}^{k} x^n \left\langle \cdot, \beta_n^{k}\right\rangle+ \sum_{n=0}^{k} x^n a_n^{k+1} \left\langle \cdot, p_{k+1}\right\rangle+ x^{k+1} a_{k+1}^{k+1} \left\langle \cdot,p_{k+1}\right\rangle\\
&=\op_{V_{k}} + \sum_{n=0}^{k+1} x^n a_{n}^{k+1} \left\langle \cdot,p_{k+1}\right\rangle\\
	&= \op_{V_{k}} + p_{k+1} \left\langle \cdot,p_{k+1}\right\rangle\quad{\text{using \eqref{eq:OPS}}} \\
&= \op_{V_{k+1}}.
\end{align*}

Case (ii). From the definition of $\beta_n^k$ given by Eq.~\eqref{eq:betb} we have that
\begin{align*}
\beta_n^{k+1}  = \begin{cases}
	\displaystyle \beta_n^k+a^{k+1}_{(k-n+1)/2}\ p_{k+1} & \text{for $k-n+1$ even},\\
	\beta_n^{k} & \text{for $k-n+1$ odd},\\
\beta_{k+1}^{k+1}= a^{k+1}_0\, p_{k+1} & \text{for $n=k+1$}.
\end{cases}
\end{align*}
For $n \in \N$ let us define two sets of numbers
\begin{align*}
{\cal{S}}_e&= \{n\le k\ : \ k-n+1\ \text{even}\},\quad\text{and}\quad
{\cal{S}}_o= \{n\le k \ : \ k-n+1\  \text{odd}\}.
\end{align*}
Then,
\begin{align*}
\sum_{n=0}^{k+1} x^n \left\langle \cdot, \beta_n^{k+1} \right\rangle&=\sum_{n=0}^{k} x^n \left\langle \cdot, \beta_n^{k+1} \right\rangle+ x^{k+1}\left\langle \cdot, \beta_{k+1}^{k+1}\right\rangle = \sum_{n \in {\cal{S}}_o} x^n \left\langle \cdot, \beta_n^{k+1} \right\rangle\\
&+\sum_{n \in {\cal{S}}_e} x^n \left\langle \cdot, \beta_n^{k+1} \right\rangle+ x^{k+1} \left\langle \cdot, \beta_{k+1}^{k+1}\right\rangle \\
& = \sum_{n=0}^{k} x^n \left\langle \cdot, \beta_n^{k} \right\rangle+ \sum_{n \in {\cal{S}}_e}x^n \left\langle \cdot, a_{(k-n+1)/2}^{k+1} p_{k+1}\right\rangle+ x^{k+1} \left\langle \cdot , a_0 ^{k+1} p_{k+1}\right\rangle\\ 
& = \op_{V_{k}} + \sum_{i=0}^{\lfloor k/2\rfloor} x^{k+1-2i} a_i^{k+1} \left\langle \cdot, p_{k+1}\right\rangle= \op_{V_{k}} +  p_{k+1} \left\langle \cdot , p_{k+1}\right\rangle = \op_{V_{k+1}}.
\end{align*}
\end{proof}
\begin{remark}
It is worth noticing that Property~\ref{prop:self} implies that 
\begin{align*}
	\op_{V_k}= \sum_{i=0}^k x^i \left\langle \cdot, \beta_i^k \right\rangle = 
\sum_{i=0}^k \beta_i^k  \left\langle \cdot ,  x^i\right\rangle .
\end{align*}
\end{remark}
\begin{property}
The set $\{\beta^k_n\}_{n=0}^k$ defined in~\eqref{eq:bet} or in~\eqref{eq:betb} is biorthogonal to the set of monomials $\{x^n\}_{n=0}^k$ on the interval $x \in \In\subseteq \R$, i.e., 
\begin{align*}
	\left\langle \beta_n^k, x^m\right\rangle =\delta_{n,m}.
\end{align*}
\end{property}
\begin{proof}
We consider only case (i), in which we will prove the result by induction. The proof of case (ii) follows a similar approach. For $n=k=0$ the property holds, because $\beta^0_0=a_0^0 p_0$, $x^0=p_0/a_0^0$, and  $\left\langle p_0, p_0\right\rangle =1$. For the induction step within the proof, we assume that
\begin{align}\label{eq:asu}
\left\langle \beta_n^k, x^m\right\rangle =\delta_{n,m},\quad \text{for $n=0,\dots,k$ and $m=0,\dots,k$}.
\end{align}
We must consider four specific cases separately, to show that 
\begin{align*}
	\left\langle \beta_n^{k+1}, x^m\right\rangle =\delta_{n,m}, \quad\text{for $n=0,\dots,k+1$ and $m=0,\dots,k+1$}.
\end{align*}
The four cases to consider are
\begin{enumerate}
\item[(i)] $n=0,\dots, k$ and $m=0,\dots, k$.
\item[(ii)] $n=k+1$ and $m=0,\dots,k$.
\item[(iii)] $n=0,\dots,k$ and $m=k+1$.
\item[(iv)] $n=k+1$ and $m=k+1$.
\end{enumerate}	

\noindent Case (i): Using~\eqref{eq:bet} we write
\begin{align*}
	\left\langle \beta_n^{k+1}, x^m\right\rangle= \left\langle \beta_n^k, x^m\right\rangle+ a_{n}\left\langle p_{k+1},  x^m\right\rangle ,
\end{align*}
so that from the assumption $\left\langle \beta_n^k, x^m\right\rangle =\delta_{n,m}$, with $n=0,\dots,k$ and $m=0,\dots,k$, and Eq.~\eqref{eq:ipxp}, it follows that $\left\langle \beta_n^{k+1}, x^m\right\rangle = \delta_{n,m}$ for $n=0,\dots,k$ and $m=0,\dots,k$.\\

\noindent Case (ii):  $\left\langle \beta_{k+1}^{k+1}, x^m\right\rangle = a_{k+1}^{k+1}\left\langle p_{k+1},  x^m\right\rangle =0$ from \eqref{eq:ipxp}, since $m<k+1$. \\

\noindent Case (iii): Using Eqs.~\eqref{eq:bet},~\eqref{eq:xp}, and~\eqref{eq:asu} we have
\begin{align*}
\left\langle \beta_n^{k+1}, x^{k+1}\right\rangle&=\left\langle \beta_n^{k}, x^{k+1}\right\rangle+ a_{n}^{k+1} \left\langle p_{k+1}, x^{k+1}\right\rangle = \left\langle \beta_n^{k}, \frac{p_{k+1}}{a_{k+1}^{k+1}} \right\rangle
- \left\langle \beta_n^{k}, \sum_{i=0}^k \frac{a_{i}^{k+1}}{a_{k+1}^{k+1}} x^i \right\rangle+\frac{a_{n}^{k+1}}{a_{k+1}^{k+1}}\\
&= 0- \frac{a_{n}^{k+1}}{a_{k+1}^{k+1}} + \frac{a_{n}^{k+1}}{a_{k+1}^{k+1}}=0.
\end{align*}

\noindent Case (iv) The result follows by setting $m=k+1$ in Case (ii):
\begin{align*}
	\left\langle \beta_{k+1}^{k+1}, x^{k+1}\right\rangle = a_{k+1}^{k+1}\left\langle p_{k+1},  x^{k+1}\right\rangle = \frac{a_{k+1}}{a_{k+1}}=1.
\end{align*}
From the results of all four cases (i), (ii), (iii), and (iv) we can conclude by mathematical induction that Eq.~\eqref{eq:asu} holds.
\end{proof}

\subsubsection*{Application to downgrading orthogonal projectors}

\noindent Let us suppose that from knowing the orthogonal projection of $f$, $\op_{V_k} f=f_k$, one is interested in producing the orthogonal projection onto the subspace left by removing one element, $x^{\ell_j}$, from the set $\{x^n\}_{n=0}^k$, i.e., the orthogonal projection on the new subspace spanned by the set
\begin{align*}
	{\cal{D}}_{k\backslash{\ell_j}}=\{x^0,x^1,\ldots,x^{\ell_j-1}, x^{\ell_j+1}, \ldots,x^k\}.
\end{align*}
We use the notation $V_{k\backslash{\ell_j}}=\Spann\left({\cal{D}}_{k\backslash{\ell_j}}\right)$ to indicate the subspace spanned by this set. The downgrading of the projector $\op_{V_k} \to \op_{V_{k\backslash{\ell_j}}}$, is simpler using BP than having to downgrade an OPS. It can be achieved through the backward adaptive biorthogonalisation formula derived in Ref.~\cite{LRN02} and extended in Ref.~\cite{LRN07}:
\begin{align}\label{eq:betb2}
\beta_n^{k\backslash\ell_j}(x)= \beta_n^{k}(x) - \beta_{\ell_j}^{k}(x)\frac{\left\langle \beta_{\ell_j}^k, \beta_n^{k}\right\rangle}{\left\langle \beta_{\ell_j}^{k}, \beta_{\ell_j}^{k} \right\rangle}, \quad \text{for $n=0,\dots,\ell_{j-1},\ell_{j+1},\dots,k$ and $x \in \In$}.
\end{align}
\begin{property}\label{prop:appenC}
Let the set $\{\beta_n^k\}_{n=0}^k$ satisfy the following two properties:
	\begin{itemize}
\item[(i)] is biorthogonal to the set $\{x^n\}_{n=0}^k$.
\item[(ii)] provides a representation of $\op_{V_k}$ of the
	form given by Eq.~\eqref{eq:opb}, 
	\end{itemize}
then the set of functions $\{\beta_n^{k\backslash\ell_j}\}_{n=0}^k$ defined as in Eq.~\eqref{eq:betb2},  
 \begin{itemize}
	 \item[(iii)] is biorthogonal to the set ${\cal{D}}_{k\backslash{\ell_j}}$.  
	 \item[(iv)] provides a representation of $\op_{V_{k\backslash{\ell_j}}}$, the orthogonal projection onto $V_{k\backslash{\ell_j}}=\Spann\left(\{{\cal{D}}_{k\backslash{\ell_j}}\}\right)$.
\end{itemize}
\end{property}
\begin{proof} 
	The proof is presented in Appendix C. 
\end{proof}
\begin{remark}\label{rem:boomp}	
Since the approximation error arising from removing the element $x^{\ell_j}$  from the projection of $f$ onto the subspace $V_k$ given by $\op_{V_{k}} f$ is (c.f. \eqref{eq:pvnje2}) 
\begin{align*}
	\op_{V_{k}} f- \op_{V_{k\backslash{\ell_j}}}f = \frac{\beta_{\ell_j}^k \left\langle  f,  \beta_{\ell_j}^k \right\rangle}{\|\beta_{\ell_j}^k\|^2_2}.
\end{align*}
It follows that by removing  the element $x^{\ell_j}$  such that 
\begin{align*}
	\frac{\left|\left\langle  f,  \beta_{\ell_j}^k \right\rangle\right|^2}{\|\beta_{\ell_j}^k\|^2_2}\, \quad \text{is minimum}
\end{align*}
the remaining downgraded approximation minimises the new approximation error norm.
\end{remark}

\subsection{Biorthogonal polynomials from Type A classic OPS}
We illustrate our methodology by considering the particular families of biorthogonal polynomials which are generated by classical OPS. In this subsection we consider only those OPS of Type A (c.f. Eq.~\eqref{eq:OPS}). These include OPS composed of
\begin{itemize}
	\item[(i)] Legendre polynomials $\{\tilde{P}_n(x)\}_{n=0}^k$ defined on the interval $\In=[0,b]$, with inner product weight $w(x)=1$.
	\item[(ii)] Laguerre polynomials $\{L_n(x)\}_{n=0}^k$ defined on the interval $\In=[0,\infty)$, with inner product weight $w(x)= \exp(-x)$.
\end{itemize}
\subsubsection{Biorthogonal polynomials derived from Legendre polynomials}
In this case the orthogonal polynomials $\{p_n(x)\}_{n=0}^k$ that generate the biorthogonal polynomials Eq.~\eqref{eq:bet} are the Legendre polynomials $\{\tilde{P}_n(x)\}_{n=0}^k$ defined on the interval $\In=[0,b]$. These polynomials on $[0,b]$ can be obtained by the formula:
\begin{align*}
	\tilde{P}_j(x)= (-1)^j \frac{\sqrt{2j+1}}{\sqrt{b}} \sum_{i=0}^j (-1)^i \binom{j}{i} \binom{j+i}{i} \frac{x^i}{b^i},
\end{align*}
which is in the exact form defined by being Type A. Here, we note that $b^i$ indicates the $i$-th power of $b$. For this particular case, the coefficients $a_i^j$ that are required to obtain the functions $\beta_n^k$ as defined in Eq.~\eqref{eq:bet} are 
\begin{align*}
		 a_i^j= (-1)^{i+j}  \frac{\sqrt{2j+1}}{ b^i \sqrt{b}} \binom{j}{i} \binom{j+i}{i}.
\end{align*}
Accordingly, the BP generated by Legendre polynomials on $[0,b]$ adopt the form 
\begin{align*}
	\beta_n^k&= \sum_{j=n}^k \frac{{2j+1}}{{b}}(-1)^{j+n} \binom{j}{n} \binom{j+n}{n} \sum_{i=0}^j (-1)^{i+j}\binom{j}{i} \binom{j+i}{i} \frac{x^i}{b^{n+i}}\\
	&= \frac{(-1)^n}{b^{n+1}} \sum_{j=n}^k \sum_{i=0}^j (-1)^{i} \frac{{2j+1}}{b^{i}}
 \binom{j}{n} \binom{j+n}{n} \binom{j}{i} \binom{j+i}{i} x^i.
\end{align*}
Subsequently, the coefficients $c_n^k$ calculated by 
\begin{align}\label{eq:bn01}
c_n^k&=\int_0^b f(x) \beta_n^k(x) \,dx \nonumber\\
&= \frac{(-1)^n}{b^{n+1}} \sum_{j=n}^k \sum_{i=0}^j \frac{{2j+1}}{b^{i}} (-1)^{i} \binom{j}{n} \binom{j+n}{n} \binom{j}{i} \binom{j+i}{i}\int_0^b f(x) x^i \,dx, 
\end{align}
provide the orthogonal projection of $f$ onto $V_k$ as linear combination of simple monomials 
\begin{align}\label{eq:opbn01}
f_k(x)= \sum_{n=0}^k c_n^k x^n,\quad x\in [0,b].
\end{align}
When $f(x)$ is a probability density function the integrals in 
 Eq.~\eqref{eq:bn01} are the statistical moments of the probability distribution. On $V(\In,w)$ we define the $i$-th \textit{generalised} 
moment of a function as 
\begin{align}\label{eq:cn01}
\mu_i=\int_{\In} x^i f(x)\, w(x)\ dx.
\end{align}

\subsubsection*{Example 1: Polynomial approximation of noisy data using Legendre polynomials (Type A)}

We demonstrate the application of BP outlined above, using  Legendre orthogonal polynomials $\{\tilde{P}_n(x)\}_{n=0}^k$ defined on the interval $[0,b]$, for calculating the coefficients of a polynomial regression problem. In our example, we consider numerical data simulated using $N=501$ uniformly distributed points on the interval ${\cal{I}}=[0,1]$, $x_i=0.002i,\,i=0, \dots, 500$, generated by the chirp function $f_i^{\rm o}=f(x_i)= \cos(7 \pi x_i^2) + \epsilon_i$, where $\epsilon_i$ are independent Gaussian random variables with mean zero and variance equal to $0.01$. The simulated data is shown in Fig.~\ref{fig:Fig1}(a) by the green dots. Our goal is to approximate the functional form of these data points by a polynomial. As established by Theorems~\ref{thm:orthp} and~\ref{thm:biorth}, the coefficients $c_n^k$, for $n=0,\dots,k$ in 
\begin{align*}
	f_k(x)= \sum_{n=0}^k c_n^k x^n,
\end{align*}
give rise to the polynomial that minimises the $L^2$-norm of the approximation error. These can be computed using numerical integration methods, in our case, the fourth-order Simpson's 1/3 rule.


The green solid line in Fig.~\ref{fig:Fig1}(b) is the approximation obtained by fixing $k=17$. Fig.~\ref{fig:Fig1}(c) displays the downgraded version of that approximation, denoted  $f_K(x)$, and achieved by removing three terms (one by one) from $f_{k}(x)$, as prescribed in Remark~\ref{rem:boomp}. It is worth noticing that the removal of the three terms, corresponding to $x$, $x^4$, and $x^{17}$, does not significantly change the quality of the approximation. The $L^2$-norm of the error in representing the noisy chirp with the approximation $f_{k}$ given in Fig.~\ref{fig:Fig1}(b) is $4.55\times 10^{-2}$, while the downgraded approximation $f_K$ involving 15 terms (plotted in Fig.~\ref{fig:Fig1}(c)) produces $L^2$-norm error only slightly larger ($5.18\times 10^{-2}$). In order to visually highlight the convenience of the proposed  downgrading, in Fig.~\ref{fig:Fig1}(d) we show the polynomial approximation $f_{k}$ involving also $15$ terms, but composed of the first $15$ Legendre polynomials. As perceived in the figure,  this is a much worse approximation what is also reflected in the $L^2$-norm error value which increases drastically to $1.76\times 10^{-1}$. To further emphasise the comparison we now augment the downgrading to keep $K=13$ terms by removing the additional two terms $x^2$ and $x^3$. Fig.~\ref{fig:Fig1}(e) shows the resulting downgraded approximation and Fig.~\ref{fig:Fig1}(f) the approximation $f_k$, with $k=12$, composed by the first 13 Legendre polynomials. Please see the corresponding values of the $L^2$-norm errors in Table~\ref{tab:errorEx1}.

With these error values we have calculated the Bayesian information criterion (BIC) score~\cite{Pri81} which is used to assess the model selection for independent Gaussian random data, taking into account the number of parameters $\gamma$ required in the approximation,
 \begin{align*}
	\text{BIC}= \gamma \log(N)  + N \log\left(\frac{1}{N} \sum_{i=0}^{N}(f_i^{\rm o} - f_k(x_i))^2\right).
 \end{align*}
 One can interpret the BIC score as a logarithm of the efficiency between the model accuracy and its complexity, where a lower value represents a better approximation. In the above definition we have, $\gamma=18$ for approximation (b), $\gamma=15$ for approximations (c) and (d), and $\gamma=13$ for approximations (e) and (f). Even if the best BIC value amongst the three approximations is produced by model (b), model (c) with three fewer parameters would increase the BIC value by $3.59\%$. Adopting model (d) with also three fewer parameters would increase it by $83.12\%$. Since approximations with smaller (more negative) BIC are in general preferred, these scores highlight the advantage of reducing the amount of terms in the approximation by removing the three selected terms, rather than reducing the polynomial degree.

\begin{figure}[ht!]
\begin{center}
\includegraphics[width=\columnwidth]{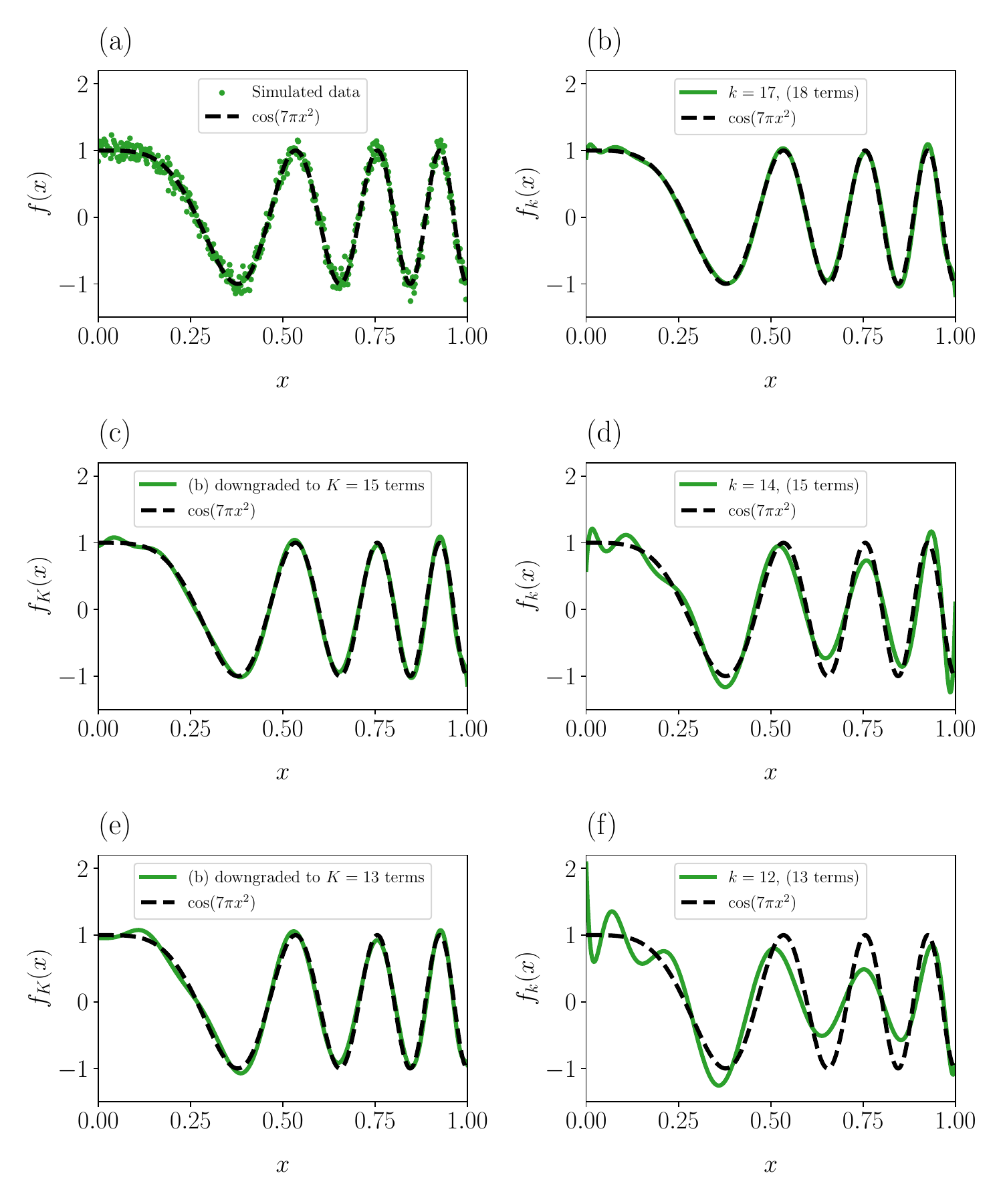}
\caption{Polynomial approximations to the simulated noisy chirp data represented as green points displayed in (a). The black dashed line in all fourth plots represent the original (noiseless) chirp function $f(x)= \cos(7 \pi x^2)$ on the interval $[0,1]$. (b) The green line corresponds to the polynomial approximation using the first $18$ Legendre polynomials corresponding to $k=17$. (c) The green line is the downgraded approximation $f_K$ to $15$ terms of the approximation displayed in figure (b) by removing three terms $x, x^4$, and $x^{17}$ following our methodology. (d) The green line is the polynomial approximation of the simulated data using the first $15$ Legendre polynomials $f_k$ with $k=14$ (the same number of terms as in (c)). Figure (e) presents the downgraded approximation $f_K$ to $13$ terms of the approximation in figure (b) by removing five terms $x, x^2, x^3, x^4$, and $x^{17}$. (f) The green line is the polynomial approximation of the simulated data using the first $13$ Legendre polynomials $f_k$ with $k=12$ (the same number of terms as in (e)).  
\label{fig:Fig1}}
\end{center}
\end{figure}

\begin{center}
	\begin{table}[htp!]
	\caption{Values of the $L^2$-norm error and the Bayesian information criterion (BIC) of the approximations $f_k$ and the downgraded approximations $f_K$ to the original simulated noisy chirp data.\label{tab:errorEx1}}
	\begin{tabular}{c|c|c}
		\centering {\bf Approximation Type} &  {\bf $L^2$-norm} & {\bf BIC}\\
		\hline
		 \centering 18-term Legendre approximation $f_k$ with $k=17$ & $4.55\times 10^{-2}$ &$-2972.37$\\
		 \hline
		 \centering 18-term Legendre approximation downgraded to $K=15$ terms $f_K$  &$5.18\times 10^{-2}$&$-2869.45$\\
		 \centering  15-term Legendre approximation $f_k$ with $k=14$  &$1.76\times 10^{-1}$ &$-1623.20$\\
		 \hline
		 \centering  18-term Legendre approximation downgraded to $K=13$ terms $f_K$ &$7.49\times 10^{-2}$ &$-2516.57$\\
		 \centering  13-term Legendre approximation $f_k$ with $k=12$ & $2.94\times 10^{-1}$ & $-1138.01$
	   \end{tabular}
	   \end{table}
\end{center}


\subsubsection{Biorthogonal polynomials from Laguerre polynomials}

In this case the orthogonal polynomials  $\{p_n(x)\}_{n=0}^k$ generating the biorthogonal polynomials Eq.~\eqref{eq:bet} are the Laguerre polynomials $\{L_n(x)\}_{n=0}^k$ in $V([0,\infty), w)$, with $w(x)= \exp(-x)$. These polynomials can be obtained by formula~\cite{Kre78}
\begin{align}\label{eq:Ln}
L_{j}(x)= \sum_{i=0}^j \binom{j}{i} \frac{(-1)^{i}}{i!} x^i.
\end{align}
Consequently, the coefficients $a_i^j$ to obtain the polynomial $\beta_n^k(x)$ through  Eq.~\eqref{eq:bet} in this particular case are
\begin{align*}
a_i^j= \binom{j}{i} \frac{(-1)^{i}}{i!}.
\end{align*}
Thus, the functions $\beta_n^k(x)$ adopt the particular form
\begin{align*}
	\beta_n^k(x)&= \sum_{j=n}^k  \binom{j}{n} \frac{(-1)^{n}}{n!}
\sum_{i=0}^j \binom{j}{i} \frac{(-1)^{i}}{i!}  x^i= \sum_{j=n}^k \sum_{i=0}^j (-1)^{n+i}\frac{1}{n!}\frac{1}{i!}
 \binom{j}{n} \binom{j}{i}   x^i,
\end{align*}
for $x\in [0,\infty)$. Subsequently, the coefficients $c_n^k$ producing the orthogonal projection of $f$ onto $V_k$ as linear combination of simple monomials 
\begin{align}
\label{Lagap}
f_k(x)= \sum_{n=0}^k c_n^k x^n,\quad x\in [0,\infty)
\end{align}
are calculated as
\begin{align}\label{eq:bnLa}
c_n^k&=\int_0^\infty f(x) \beta_n^k(x)\,e^{-x} \,dx = \sum_{j=n}^k \sum_{i=0}^j \frac{(-1)^{n+i}}{n!\,i!}\binom{j}{n} \binom{j}{i} \mu_i
\end{align}
with 
\begin{align}\label{eq:moLa}
\mu_i=\int_0^\infty x^i f(x) e^{-x} \,dx.
\end{align}

\subsubsection*{Example 2: Analytical polynomial approximation of the exponential decay}

This example aims at deriving a closed polynomial form for approximating the exponential decay function $f(x)=\exp(-\alpha x)$, for $\alpha >0$. In the first instance we use biorthogonal polynomials derived from Laguerre polynomials $\{L_n(x)\}_{n=0}^k$ over the interval $[0,\infty)$. Consequently, the generalised moments $\mu_i$  for these functions are analytically obtained as 
\begin{align*}
	\mu_i=\int_{0}^{\infty} x^{i}e^{-\alpha x} e^{-x}\,dx=\int_{0}^{\infty}  x^{i}e^{-(\alpha+1) x} \,dx=  \frac{i!}{(\alpha+1)^{i+1}}.
\end{align*}
On replacing these moments in Eq.~\eqref{eq:bnLa} we obtain:
\begin{align*}
	c_n^k&=\sum_{j=n}^k \frac{(-1)^{n}}{n!}\binom{j}{n} \sum_{i=0}^j \binom{j}{i} \frac{(-1)^i}{(\alpha+1)^{i+1}}
 = \sum_{j=n}^{k}\frac{(-1)^{n}}{n!}\binom{j}{n}\frac{\alpha^j}{(\alpha+1)^{j+1}}.
\end{align*}
Thus, our Laguerre-based biorthogonal approximation to the exponential decay has the closed form given by
\begin{align}\label{eq:expdecay}
	f_k(x) = \sum_{n=0}^{k} \sum_{j=n}^k \frac{(-1)^{n}}{n!}\binom{j}{n} \frac{\alpha^j}{(\alpha+1)^{j+1}}\, x^n,\quad 
	x\in [0,\infty).
\end{align}
The approximation ${f}_k(x)$, for $\alpha=1$ and $k=14$, is shown in Fig.~\ref{fig:Fig2}(a). We notice that for $\alpha=1$ and $x>10$ the exponential decay takes vanishing small values. More precisely, the value of the exponential at $x=10$ is $f(10)=\exp(-10)< 4.54 \times 10^{-5}$. Consequently, we find it interesting to further compare the approximation given in Eq.~\eqref{eq:expdecay} to the approximation resulting from the Legendre-based BP defined on the closed interval $[0,10]$. The latter approximation readily follows using Eq~\eqref{eq:bn01} and the moments defined by
\begin{align*}
	\mu_i=\int_{0}^{b} x^i e^{-\alpha x}\, dx=e^{-\alpha b} \sum_{j=0}^i (-1)^{i-j}\frac{i!}{j! (-\alpha)^{i-j+1}} b^j +  \frac{i!}{
		\alpha^{i+1}}
\end{align*}
with $\alpha=1$ and $b=10$. The Legendre-based BP approximation defined on $[0,10]$ is attained by placing the computed coefficients in Eq.~\eqref{eq:opbn01}. This approximation for $k=9$ is shown in Fig.~\ref{fig:Fig2}(a) (in the scale of the figure both approximations coincide with the true function). The Laguerre-based biorthogonal approximation produces the maximum absolute pointwise error, defined by $\max_{x\in [0,10]}|f(x) - f_{k}(x)|$, of $2.62 \times 10^{-4}$, while the Legendre-based biorthogonal approximation has the maximum absolute pointwise error to $f(x)$ of $2.20\times 10^{-4}$ even though involving fewer terms. The absolute pointwise errors of both approximations are shown in logarithmic scale in Fig.~\ref{fig:Fig2}(b).

We now use both polynomial types to approximate the Gamma distribution $f(x)=x \exp(-x)$, plotted in Fig.~\ref{fig:Fig2}(c). In this case, to obtain similar  Laguerre-based  and Legendre-based biorthogonal approximations we require $k= 17$ and $k=11$ respectively. The corresponding maximum absolute pointwise errors to $f(x)$ are $3.90\times 10^{-4}$ for the Laguerre-based and $8.52\times 10^{-4}$ for the Legendre-based approximation. The spatially-dependent absolute pointwise errors are shown in logarithmic scale of Fig.~\ref{fig:Fig2}(d).
\begin{figure}[ht!]
\begin{center}
	\includegraphics[width=\columnwidth]{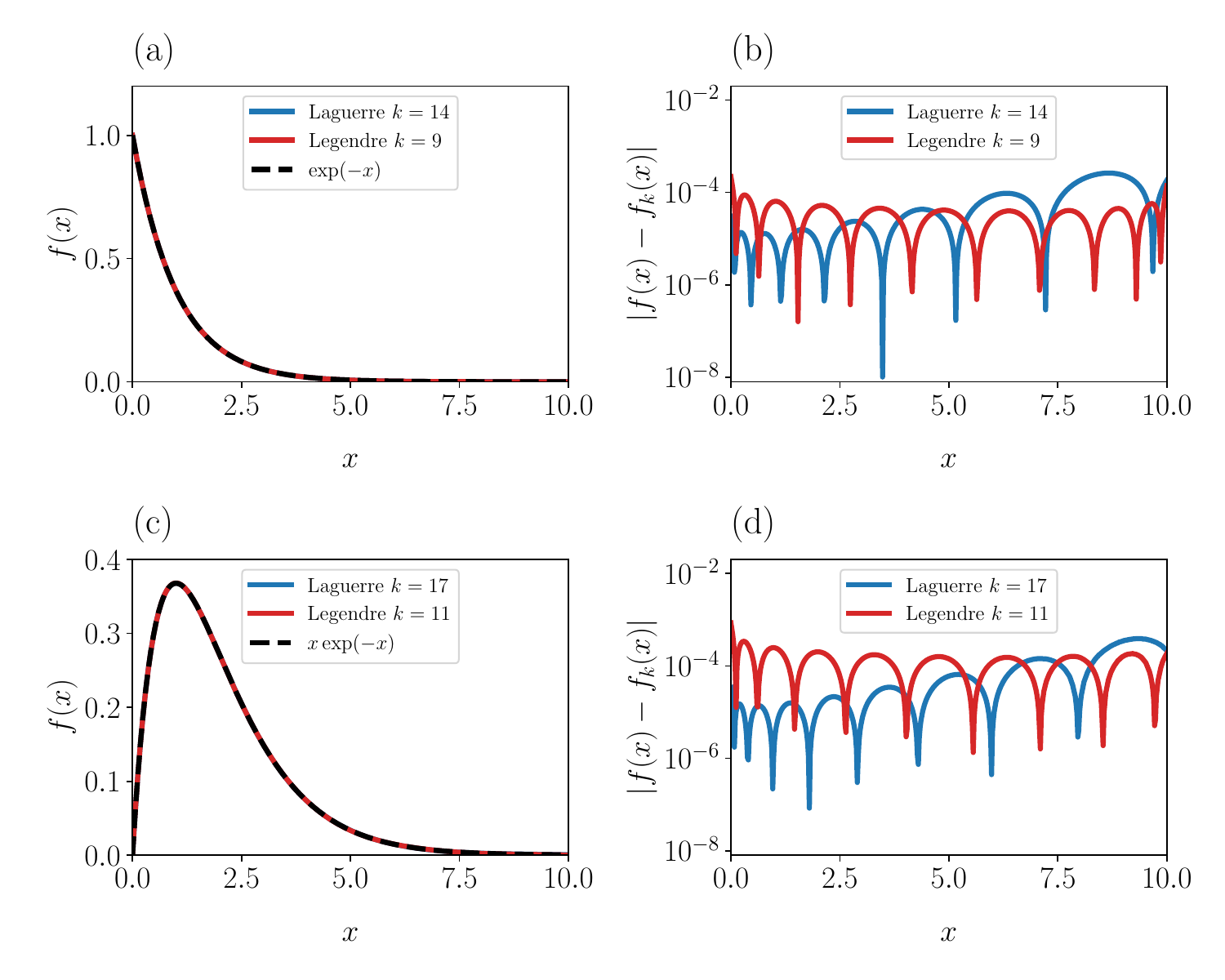}
\end{center}
	\caption{(a) The exponential decay function $f(x)=\exp(-\alpha x)$ with $\alpha=1$ is given by the black dashed curve defined on the interval $[0,10]$. Approximations using our methodology outlined in the main text using the Laguerre polynomials with $k=14$ (blue solid curve) and Legendre polynomials with $k=9$ (red solid curve) are presented (both approximations lie directly under each other). (b) Plot of the absolute pointwise errors of both approximations in logarithmic scale from figure (a). (c) The gamma distribution $f(x)=x\exp(-\alpha x)$ (black dashed curve) is approximated by Laguerre polynomials with $k=17$ (blue solid curve) and Legendre polynomials with $k=11$ (red solid curve), again with both approximations lying directly on top of each other. (d) Corresponding absolute pointwise errors of the approximations in logarithmic scale from (c).\label{fig:Fig2}}
\end{figure}

\subsection{Biorthogonal polynomials from Type B classic OPS}
We consider now families of BP which are generated by classical OPS that we refereed to as Type B (c.f.\eqref{eq:OPSB}). These are
\begin{itemize}
        \item[(i)] Legendre polynomials $\{P_n(x)\}_{n=0}^k$ corresponding to $\In=[-1,1]$ and $w(x)=1$.
	\item[(ii)] Chebyshev polynomials $\{T_n(x)\}_{n=0}^k$  corresponding to $\In=[-1,1]$ and $w(x)= 1/\sqrt{1-x^2}.$	
\end{itemize}

\subsubsection*{Biorthogonal polynomials from Legendre and Chebyshev polynomials on  $\In=[-1,1]$.}

(i) Legendre polynomials $\{P_n(x)\}_{n=0}^k$  on $\In=[-1,1]$ are generated as Ref.~\cite{Kre78}.

\begin{align*}
	P_j(x)=\frac{\sqrt{2j+1}}{2^j\sqrt{2}}\sum_{l=0}^{\lfloor j/2 \rfloor}(-1)^l \binom{j}{l} \binom{2j-2l}{j}x^{j-2l}.
\end{align*}

(ii) Chebyshev polynomials $\{T_n(x)\}_{n=0}^k$  on $\In=[-1,1]$ are generated as Ref.~\cite{Kre78}
\begin{align*}
T_0(x)&=\frac{1}{\sqrt{\pi}},\qquad
T_j(x)=\frac{\sqrt{2}}{\sqrt{\pi}}\frac{j}{2}\sum_{l=0}^{\lfloor j/2 \rfloor}
(-1)^l  2^{j-2l} \frac{(j-l-1)!}{l!(j-2l)!}x^{j-2l}\,\quad\text{for $j =1, 2, 3, \dots$}.
\end{align*}

Then, for  case (i)
\begin{align}\label{eq:P-11}
P_{2i+n}(x)=\frac{\sqrt{4i+2n+1}}{2^{2i+n}\sqrt{2}}\sum_{l=0}^{\lfloor (2i+n)/2 \rfloor}(-1)^l \binom{2i+n}{l} \binom{2(2i+n)-2l}{2i+n}x^{2i+n-2l}
\end{align}
and
\begin{align}\label{eq:a-11}
a_l^{2i+n}= (-1)^l\frac{\sqrt{4i+2n+1}}{2^{2i+n}\sqrt{2}} \binom{2i+n}{l} \binom{2(2i+n)-2l}{2i+n}.
\end{align}
The BP arise using Eq.~\eqref{eq:P-11} and Eq.~\eqref{eq:a-11} in the equation below
\begin{align}\label{eq:beta}
\beta_n^k(x)&= \sum_{i=0}^{\lfloor (k-n)/2 \rfloor} a_i^{2i+n} P_{2i+n}(x),\quad x \in [-1,1].
\end{align}

For case (ii)
\begin{align}\label{eq:T-11}
T_{0}&=\frac{1}{\sqrt{\pi}}, \quad
T_{2i+n}(x)=\frac{\sqrt{2}}{\sqrt{\pi}}\frac{2i+n}{2}\sum_{l=0}^{\lfloor (2i+n)/2 \rfloor}(-1)^l  2^{2i+n -2l} \frac{(2i+n-l-1)!}{l!(2i+n-2l)!}x^{2i+n-2l}
\end{align}
so that
\begin{align}\label{eq:ac-11}
{a'}_0^0&=\frac{1}{\sqrt{\pi}},\qquad
{a'}_l^{2i+n}= (-1)^l \frac{\sqrt{2}}{\sqrt{\pi}}\frac{2i+n}{2}2^{2i+n -2l} \frac{(2i+n-l-1)!}{l!(2i+n-2l)!} 
\end{align}
and the concomitant BP are obtained using in Eq.~\eqref{eq:beta} the polynomials~\eqref{eq:T-11} and coefficients \eqref{eq:ac-11}, instead of \eqref{eq:P-11} and \eqref{eq:a-11}.

\subsubsection*{Example 3: Sparse polynomial approximation}

In this example we illustrate the construction of sparse polynomial approximations. Let us consider for instance a given $f(x) \in {\cal{P}}_{19}$, where ${\cal{P}}_{19}$ indicates the set of polynomial of degree at most $19$ on $[-1,1]$. Thus,  $f(x)$ can be expressed as in (i) and (ii) below.
\begin{itemize}
 \item[(i)] A linear combination of Legendre polynomials
\begin{align*}
	g(x)=d_0 P_0 + d_1 P_1(x)  + d_2 P_2(x) + \cdots + d_{19} P_{19}(x),\quad x\in [-1,1],
\end{align*}
with $d_n= \left\langle P_n, f \right\rangle$, for $n=0, \dots, 19$.
\item[(ii)] A linear combination of monomials  
\begin{align*}
	h(x)=c_0^{19} x^0 + c_1^{19} x + c_2^{19} x^2 + \cdots + c_{19}^{19} x^{19},\quad x\in [-1,1],
\end{align*}
with $c_n^{19}= \left\langle \beta_n^{19}, f \right\rangle$, for  $n=0,\dots, 19$.
\end{itemize}
Both representations (i) and (ii) of $f(x),\, x\in [-1,1]$  are equivalent when all the coefficients in (i) and (ii) are included. However, when considering $K$-term approximations of $f(x)$, with at most $K$ non-zero coefficients, the results are only equivalent if the subspaces where the approximations lie are equal. This in general does not happen, because the approximating subspaces giving the best approximation in each case are in general different.

For case (i) the $K$-term approximation $g_K(x)$ of $f(x)$, which is optimal in the sense of minimising the norm of the residual error, is constructed simply by keeping the $K$ coefficients $d_{\ell_i}$, for $i=1,\ldots,K$ of the largest magnitude out of $|d_n|$ for $n=0,\dots, 19$. For case (ii) the search for the $K$-term approximation $h_K(x)$ leaving the residual  error  of minimal norm becomes computationally demanding, as there are $20!/(20! (20- K)!)$ possibilities to be considered.  Hence, for case (ii) we adopt an approximation $h_K(x)$ of $f(x)$ which is step-wise optimal. This is achieved by sequentially removing  terms from $h(x)$ as indicated in Remark~\ref{rem:boomp}.

In order to test both approaches we generate 50 polynomials
\begin{align*}
	f(x)= a_0 + a_1 x  + a_2 x^2 + \cdots + a_{19} x^{19},\quad x\in [-1,1],
\end{align*}
where the coefficients $a_n$, for $n=0,\dots,19$ are independent and randomly generated from a zero mean, unit variance, normal distribution. 

In the first instance to facilitate the visual differentiation of each case we construct the $K=6$ term approximation for case (i) and (ii).  The olive green dots in Fig.~\ref{fig:Fig3} correspond to the $L^2$-norm error by the approximation with respect to method (i), for each of the 50 randomly generated polynomials. The purple dots depict the $L^2$-norm error for method (ii). For visualisation purposes, the two errors for each realisation are connected by a straight vertical line for clear comparison, and colour coded to indicate which method is more accurate.

\begin{figure}[htp!]
\begin{center}
\includegraphics[width=0.75\columnwidth]{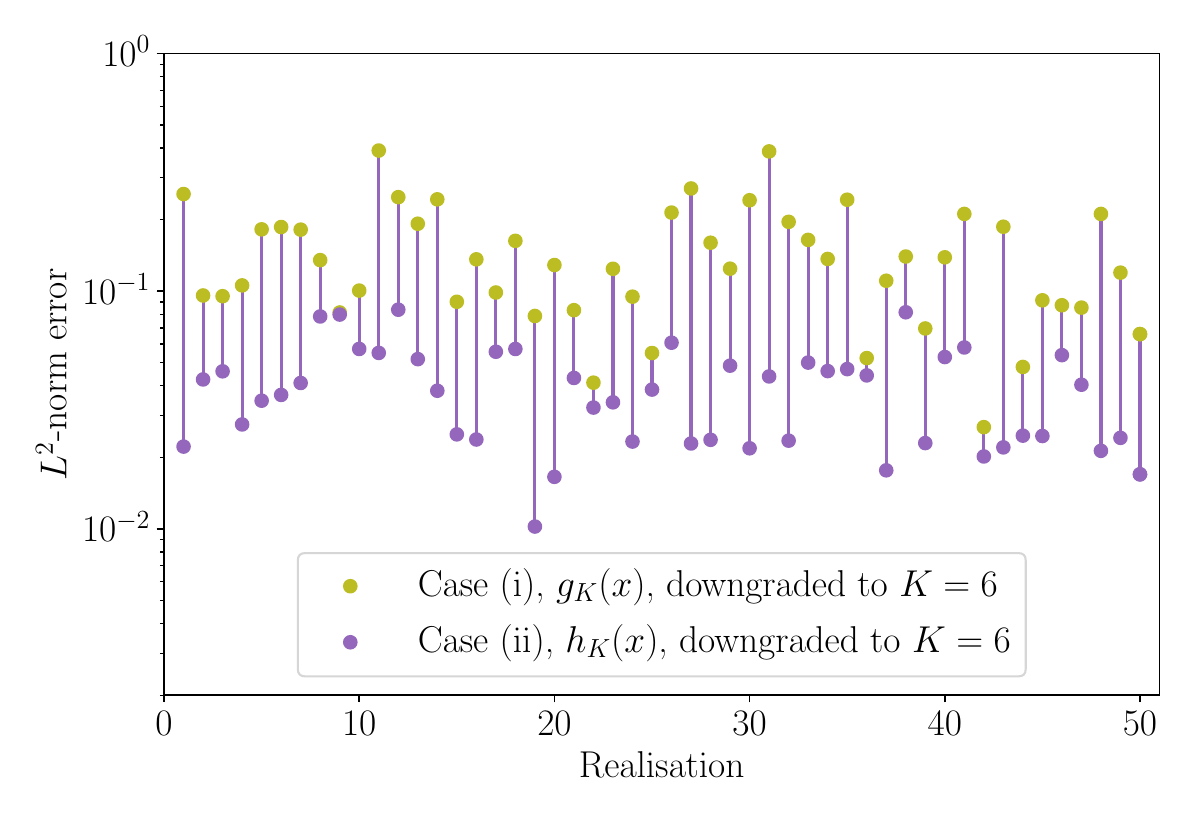}
\end{center}
\caption{The $L^2$-norm error norm, in logarithmic scale, for the downgraded polynomial approximation for each of the 50 random realisations of the polynomial $f(x)$. The olive green dots correspond to the $6$-term approximation constructed using the methodology of (i). The purple dots correspond to the $6$-term approximation for case (ii).}
\label{fig:Fig3}
\end{figure}

The mean $L^2$-norm error for case (i) is $1.47\times 10^{-1}$ with standard deviation $7.95\times 10^{-2}$ while the mean $L^2$-norm error for case (ii) is $3.46\times 10^{-2}$ with standard deviation $1.81\times 10^{-2}$.


The approximations to polynomial $f(x)$ for one realisation, using the two $K=6$-term approximations corresponding to cases (i) and (ii) are shown in Fig.~\ref{fig:Fig4}(a), with  Fig.~\ref{fig:Fig4}(b) showing the corresponding absolute pointwise errors in logarithmic scale.

\begin{figure}[htp!]
\begin{center}
\includegraphics[width=\columnwidth]{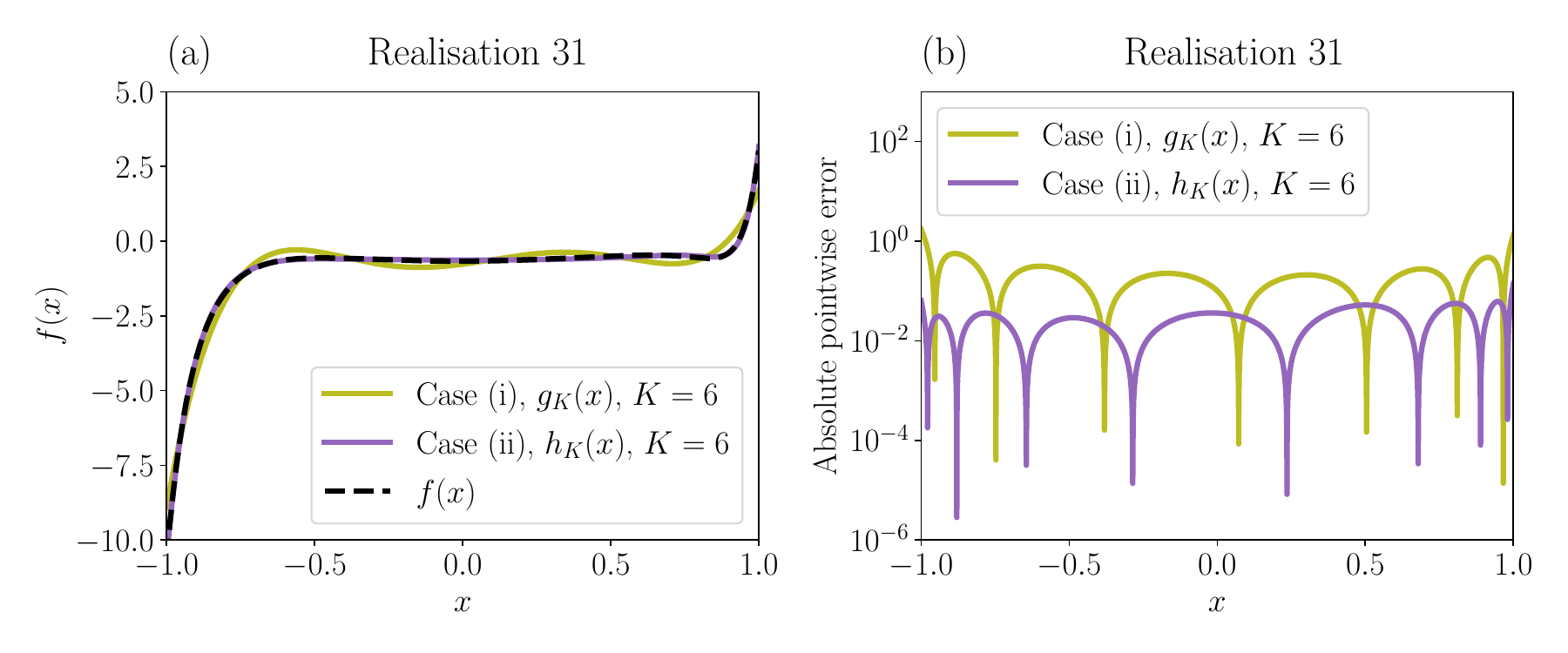}
\end{center}
\caption{Example of the approximations from realisation 31. (a) The approximation of a randomly generated polynomial $f(x)$ of degree at most 19 on the interval $[-1,1]$ is shown by the black dashed curve. The $K=6$-term approximation $g_K$ by Legendre orthogonal polynomials (case (i)) is represented by the olive green solid curve. The $K=6$-term approximation $h_K$ achieved by the biorthogonal polynomials (case (ii)) is represented by the purple solid curve indistinguishable from the black dashed curve for $f(x)$. (b) Absolute pointwise error of the two approximations to the randomly generated polynomial $f(x)$. The olive green curve is the error associated to $g_K$ (case (i)) and the purple curve is the error associated to $h_K$ (case (ii)).\label{fig:Fig4}}
\end{figure}

In order to demonstrate the difference in quality of the case (i) and case (ii) approximations in a range of smaller errors,  we repeat the downgrading test for $K=7,\ldots,19$ and compute the mean $L^2$-norm and mean maximum absolute pointwise error across all $50$ realisations for both cases. The results are depicted in Fig.~\ref{fig:Fig5}. We observe that for all the range of $K$ values case (ii) yields more accurate approximations.

\begin{figure}[htp!]
        \begin{center}
        \includegraphics[width=\columnwidth]{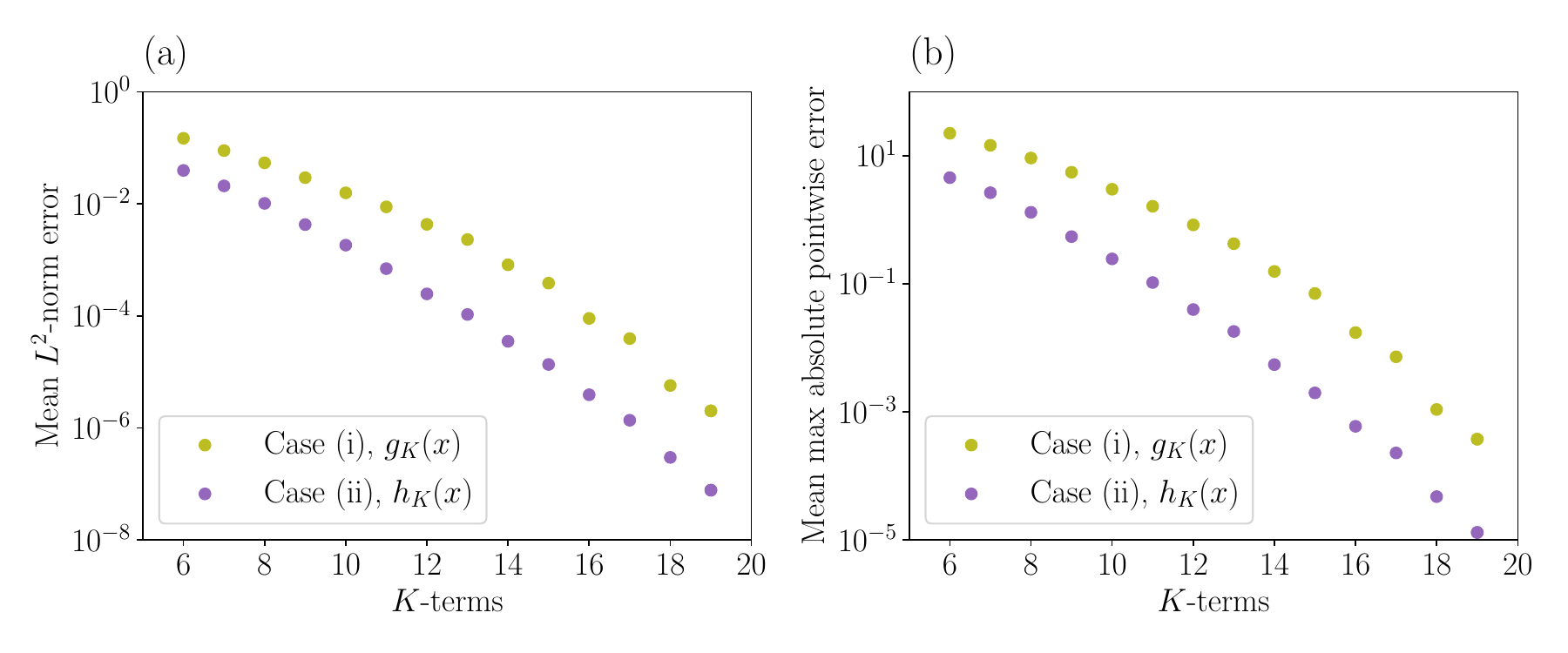}
        \end{center}
        \caption{The mean $L^2$-norm error norm (a) and mean maximum absolute pointwise error (b), in semi-logarithmic scale, averaged across all $50$ random realisations for the downgraded polynomial approximations computed via case (i) and case (ii) respectively. The mean errors are plotted versus the level of downgrading, i.e., the number of terms $K=6,\ldots,19$ in the approximations. We observe in all scenarios, case (ii) yields more accurate approximations.}
        \label{fig:Fig5}
        \end{figure}

It is clear from this numerical experiment that when the aim is to produce sparse approximation of polynomials the biorthogonal framework outlined in this work may result in more accurate approximations. This is because an orthogonal representation of polynomials is realised for mathematical and numerical convenience in posterior calculations. However, it is well known (see~\cite{RNA16,RNC19,CRN21} for some examples) that in general, releasing the condition of orthogonality in producing $K$-term approximations may introduce significant gains in terms of sparsity. 

\subsubsection*{Example 4: Analytical polynomial approximations using Legendre and Chebyshev polynomials (Type B) with comparison to matrix inversion methods}

This example applies the Legendre and Chebyshev based classic BP to compute the coefficients in the polynomial approximating of the continuous function:
\begin{align}\label{eq:chi2}
f(x)= (1-x^2)e^{-x}\sin(8 \pi x), \quad \text{for $x\in[-1,1]$}.
\end{align}
Approximating this function up to an $L^2$-norm error of $10^{-4}$, using either Legendre or Chebyshev polynomials, requires to consider a polynomial approximation of degree $k=36$. The function and the approximations are represented in Fig.~\ref{fig:Fig6}(a). Figure~\ref{fig:Fig6}(b) shows the absolute pointwise error  of the approximations in logarithmic scale. 

\begin{figure}[htp!]
\begin{center}
\includegraphics[width=\columnwidth]{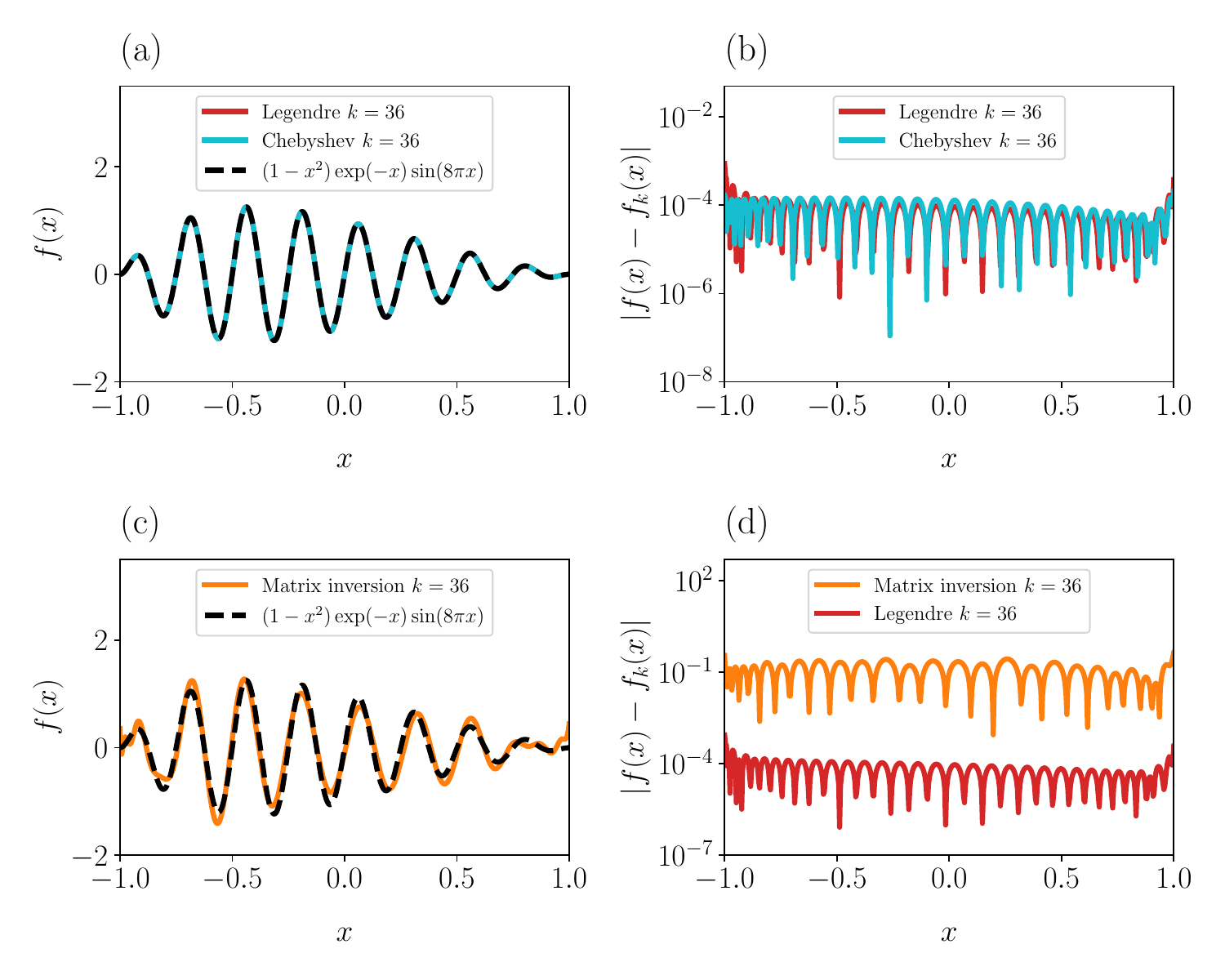}
\caption{(a) Approximations using our methodology outlined in the main text using the Legendre polynomials with $k=36$ (red solid curve) and Chebyshev polynomials with $k=36$ (cyan solid curve) are presented (both approximations lie directly under each other) to the function $f(x)=(1-x^2)\exp(- x)\sin(8\pi x)$ (Eq.~\eqref{eq:chi2}) given by the black dashed curve on the interval $[-1,1]$. (b) Plot of the absolute pointwise error of the approximations in logarithmic scale from (a). (c) Plot of the approximation of the function~\eqref{eq:chi2} by inverting the Gram matrix using a polynomial of degree $k=36$ (orange solid curve). (d) The absolute pointwise error arising from the polynomial produced by inverting the Gram matrix (solid orange curve) compared to the approximation given in (a) by applying our methodology using Legendre polynomials with $k=36$ (solid red line) which is three orders of magnitude smaller on average.\label{fig:Fig6}}
\end{center}
\end{figure}

\begin{remark}
It is worth commenting that the polynomial approximation of the functions \eqref{eq:chi2} can not be realised by the standard way of solving the least squares problem. Indeed, the standard method would  try to solve the normal equations: 
\begin{align*}
	\left\langle f , x^j \right\rangle= \sum_{n=0}^k c_n^k \left\langle x^n , x^j\right\rangle ,\quad \text{for $j=0, \dots,k$},
\end{align*}
which introduces the need to invert the Gram matrix of elements 
	\begin{align*}
		G(n,j)=  \left\langle x^n , x^j\right\rangle = \frac{1- (-1)^{n+j+1}}{n+j+1},\quad \text{for $n=0, \dots, k$ and $j=0, \dots, k$}.
	\end{align*}
This is an Hankel-type matrix which is very badly conditioned even for small values of $k$~\cite{ZT01}. In particular for $k=36$ the determinant if close to zero and the condition number is $10^{18}$. Figs.~\ref{fig:Fig6}(c) and (d) show the functional approximation and absolute pointwise error when computing using the matrix inversion method respectively. Comparison to the error produced using our biorthogonal polynomials via Legendre polynomials is presented in Fig.~\ref{fig:Fig6}(d).

It is worth noticing that orthogonal polynomials can also be used as a pre-conditioner to decrease the condition number of the matrix to be inverted.  Indeed, using Theorems 1 and 2 and letting $Q_k,\,k=1,\ldots,36$  be either Legendre or  Chebyshev polynomials to represent the orthogonal projector, the least squares solution to the approximations by a polynomial of degree $36$ can be expressed as 
\begin{align*}
\op_{V_{36}} f(x) &= \op_{V_{36}} \sum_{n=0}^{36} c_n x^n   = \sum_{k=0}^{36} \sum_{n=0}^{36} c_n 
	 Q_k(x)  \left \langle Q_k, x^n\right\rangle,
\end{align*}
or equivalently, after taking an inner product on both sides with 
each $Q_k$,
\begin{align*}
\left \langle Q_k , f \right\rangle= \sum_{n=0}^{36} 
	\left  \langle Q_k, x^n \right \rangle  c_n, \quad \text{for $k=0,\dots, 36$}.
\end{align*}	
From~\eqref{eq:ipxp} we gather that the new matrix to be inverted, of elements $\left \langle Q_k, x^n\right\rangle$, is lower triangular. Its inverse is more stable than that of the Gram matrix and produces equivalent results to those achieved by the proposed biorthogonalisation.
	


\end{remark}

%

\section{Conclusions}\label{sec:conclusions}

A  recursive procedure for constructing dedicated sets of polynomials $\{\beta_n^k\}_{n=0}^k,\, k \in \N$ in inner product spaces $V(\In, w)$ was proposed. Each set fulfils the following properties:
\begin{itemize}
	\item [(i)] It is biorthogonal to the basis of monomials $\{x^n\}_{n=0}^k$ in $V(\In, w)$. 
\item [(ii)] Span the same subspace  $V_k$ as $\{x^n\}_{n=0}^k,\, k \in \N$. 
\end{itemize}
These two properties amount to a representation of the orthogonal projection operation onto $V_k$, of an element $f\in  V(\In, w)$,  as given by
\begin{align*}
	\op_{V_k} f = \sum_{n=0}^k x^n \left\langle f, \beta_n^k \right\rangle.
\end{align*}
Since the  proposed construction is based on classic orthogonal polynomials there is no matrix inversion involved in the process. This feature is particularly convenient to decide the degree of a polynomial model by adjusting the residual error. Additionally, it is straightforward to realise a posterior reduction of the model by the removal of any monomial from the basis.

The procedure was illustrated by deriving biorthogonal polynomials based on  Legendre, Laguerre and Chebyshev polynomials. The consideration of other orthogonal polynomials, such as Hermite polynomials, would readily follow.

The convenience of the proposal was highlighted by four numerical examples: (i) Polynomial regression modelling of noisy data. (ii) Derivation of analytic polynomial forms for the exponential decay function. (iii) Sparse polynomial approximation. (iv) The approximation of a continuous function which requires a polynomial of degree at least $k=36$ to be well represented. Such an example, which cannot not be numerically realised by standard techniques involving the inversion of a Hankel-type matrix, gives an inside of the many applications that could benefit from the proposed construction of biorthogonal polynomials.

\section*{Acknowledgement}
The authors thank Dr Juan Neirotti for fruitful discussions on this work. JL was supported by the European Union’s Horizon 2020 research and innovation program under the Marie Skłodowska-Curie Grant Agreement No. 823937 for the RISE project HALT and the Leverhulme Trust Project Grant No. RPG-2021-014.

\appendix

\section{Proof of Theorem~\ref{thm:orthp}}
\begin{proof}
Recall that the optimal approximation that minimises the $L^2$-norm in $V_k$ for any $f \in V$ satisfies
\begin{align*}
        \|f - f^\ast\|_2= \min_{f_k \in V_k}\| f - f_k\|_2.
\end{align*}
For any $f_k \in  V_k$  we can write that $f = f + \op_{V_k} f - \op_{V_k} f$  to express $\| f - f_k\|_2^2$ in a convenient form
\begin{align*}
\| f - f_k \|_2^2 &= \left\langle f - f_k, f - f_k \right\rangle= \left\langle f - f_k + \op_{V_k} f -  \op_{V_k} f, f - f_k + \op_{V_k} f -  \op_{V_k} f \right\rangle.
\end{align*}
Since $\left\langle f_k, f - \op_{V_k} f \right\rangle = \left\langle f_k, \op_{V_k^\perp}f \right\rangle =0$ we have
\begin{align*}
\|f - f_k \|_2^2 = \left\langle f - \op_{V_kf},
 f - \op_{V_k}f \right\rangle+ \left\langle \op_{V_k} f- f_k,   \op_{V_k} f - f_k \right\rangle .
\end{align*}
i.e.
\begin{align*}
\| f - f_k \|_2^2= \| f - \op_{V_k} f\|_2^2 + \| \op_{V_k} f- f_k \|_2^2.
\end{align*}
Hence, $\| f - f_k \|_2$ is minimised when $f^\ast=f_k=\op_{V_k} f$.
\end{proof}
\section{Proof of Theorem~\ref{thm:oop}}
\begin{proof}
We need to show that $\op_{V_k}$ as defined in Eq.~\eqref{eq:opob} fulfils
\begin{itemize}
        \item [(i)] $\op_{V_k}^2 = \op_{V_k}$.
        \item [(ii)]$\op_{V_k}f_k = f_k$ for all $f_k \in V_k$  and  $\op_{V_k}f_k^\perp=0$ for all $f_k^\perp \in V_k^\perp$.
\end{itemize}
Since $\left\langle p_j, p_i \right\rangle=\delta_{j,i}$ it follows from Eq.~\eqref{eq:opob}, that
\begin{align*}
        \op_{V_k}^2 = \op_{V_k}\op_{V_k}=\sum_{j=0}^k \sum_{i=0}^k p_i \left\langle p_j , p_i\right\rangle \left\langle \cdot, p_j \right\rangle= \sum_{j=0}^k p_j  \left\langle \cdot, p_j\right\rangle = \op_{V_k},
\end{align*}
which proves (i).

For (ii) we use $\{p_n\}_{n=0}^k$ to  write an arbitrary vector in $V_k$ as the linear combination $\displaystyle{f_k= \sum_{n=0}^k s_n p_n}$, for some real coefficients $s_n$ for $n=0,\dots,k$. Then
\begin{align*}
        \op_{V_k}f_k =  \sum_{n=0}^k \sum_{i=0}^n p_i \left\langle s_n p_n, p_i\right\rangle=\sum_{n=0}^k \sum_{i=0}^n s_np_i\delta_{n,i} =\sum_{n=0}^k  s_n p_n= f_k.
\end{align*}

Since, $f_k^\perp \in V_k^\perp$ we have that $\left\langle f_k^\perp, p_n \right\rangle=0$ for $n=0,\dots,k$, and it follows that $\sum_{n=0}^k p_n \left\langle f_k^\perp , p_n\right\rangle =0$. Subsequently, we conclude that for $f \in V(\In, w)$ the operation $\op_{V_k}f=\sum_{n=0}^k p_n \left\langle f, p_n \right\rangle$ is the orthogonal projection of $f$ onto $V_k(\In, w)$.
\end{proof}

\section{Proof of Property~\ref{prop:appenC}}
This proof follows the steps required to prove Theorem 3.1 in Ref.~\cite{LRN02}. Assuming properties (i) and (ii) stated in Property~\ref{prop:appenC}, we will now present proofs of properties (iii) and (iv).

The proof of (iii) in Property~\ref{prop:appenC} is a direct consequence of the assumed biorthogonality property $\left\langle x^m, \beta_n^k\right\rangle= \delta_{n,m}$ for $m=0,\dots,k$ and $n=0,\dots,k$.

To prove statement (iv) in Property~\ref{prop:appenC}, we use
\begin{align}\label{eq:duba}
\beta_n^{k\backslash\ell_j}(x)= \beta_n^{k}(x) - \beta_{\ell_j}^{k}(x)\frac{\left\langle \beta_{\ell_j}^k, \beta_n^{k}\right\rangle}{\left\langle \beta_{\ell_j}^{k}, \beta_{\ell_j}^{k} \right\rangle}, \quad n=0,\ldots,\ell_{j-1},\ell_{j+1},\ldots,k, \quad x \in \In,
\end{align}
to write
\begin{align}\label{eq:pvnje}
\op_{V_{k\backslash\ell_j}}= \sumnj x^n \left\langle \cdot, \beta_n^{k\backslash\ell_j} \right\rangle= \sumnj x^n  \left\langle \cdot,  \beta_n^k \right\rangle - \sumnj   x^{n} \left\langle \cdot, \beta_{\ell_j}^k\right\rangle\frac{\left\langle \beta_{\ell_j}^k, \beta_n^{k}\right\rangle}{\left\langle \beta_{\ell_j}^{k}, \beta_{\ell_j}^{k} \right\rangle}.
\end{align}
This will enable us to prove the following statements:
\begin{itemize}
	\item[(a)] $\op_{V_{k\backslash\ell_j}} g = g$ for all $g \in V_{k\backslash\ell_j}$.
	\item[(b)] $\op_{V_{k\backslash\ell_j}} g = 0$ for all $g \in V_{k\backslash\ell_j}^\perp$, with $V_{k\backslash\ell_j}^\perp$ denoting the orthogonal complement of $V_{k\backslash\ell_j}$ in $V_k$. 
\end{itemize}
Since every $g \in V_{k\backslash\ell_j}$ can be expressed as a linear combination $g = \displaystyle \sumnj a_n x^n$ for some coefficients $a_n$ for $n=0,\dots, j-1, j+1, \dots, N$, then  using  $\left\langle x^l, \beta_n^k\right\rangle= \delta_{n,l}$ for $l=0,\dots,k$, and $n=0,\dots, k$ we have:
\begin{align*}
	 \op_{V_{k\backslash\ell_j}}= \sumnj x^n \sumnl a_l \left\langle  x^l , \beta_n^k \right\rangle+0=\sumnj x^n  a_n= g,
\end{align*}
which proves (a). Moreover, notice that since 
\begin{align*}
	\sumnj x^n \left\langle \cdot,  \beta_n^k \right\rangle = \op_{V_k}- x^{\ell_j} \left\langle  \cdot,  \beta_{\ell_j}^k \right\rangle ,
\end{align*}
Eq.~\eqref{eq:pvnje} can be re-written as
\begin{align}\label{eq:pvnje2}
\op_{V_{k\backslash\ell_j}}&= \op_{V_k} - x^{\ell_j} \left\langle  \cdot,  \beta_{\ell_j}^k \right\rangle- \frac{\op_{V_k} \beta_{\ell_j}^k \left\langle  \cdot,  \beta_{\ell_j}^k \right\rangle}{
\|\beta_{\ell_j}^k\|^2_2} + {x^{\ell_j} \left\langle  \cdot  , \beta_{\ell_j}^k \right\rangle}\nonumber\\
&=  \op_{V_k} -  \frac{\op_{V_k} \beta_{\ell_j}^k \left\langle  \cdot,  \beta_{\ell_j}^k \right\rangle}{\|\beta_{\ell_j}^k\|^2_2} =  \op_{V_k} - \frac{\beta_{\ell_j}^k \left\langle  \cdot,  \beta_{\ell_j}^k \right\rangle}{\|\beta_{\ell_j}^k\|^2_2}.
\end{align}
Since $V_{k\backslash\ell_j}^\perp$ is spanned by the single element $\beta_{\ell_j}^k \in V_k$ for all $g \in V_{k\backslash\ell_j}^\perp = \alpha \beta_{\ell_j}^k,$ for some constant $\alpha  \in \R$. Then 
\begin{align*}
	\op_{V_{k\backslash\ell_j}}  \alpha \beta_{\ell_j}^k=\alpha  \beta_{\ell_j}^k  -  \alpha  \beta_{\ell_j}^k=0,
\end{align*}
which demonstrates (b) and completes the proof of Property~\ref{prop:appenC}.




\end{document}